\documentclass[12pt]{amsart}
\usepackage{amsfonts}
\usepackage{amsmath}
\usepackage{amssymb,latexsym}
\usepackage[mathcal]{eucal}
\usepackage{amscd}

\usepackage[pdftex,bookmarks,colorlinks,breaklinks]{hyperref}
\input xy
\xyoption{all}

\newcommand{\Bcal}{\mathcal{B}}

\newcommand{\Ucal}{\mathcal{U}}

\newcommand{\ch}{\mathbf{1}}

\newcommand{\Z}{\mathbb{Z}}

\newcommand{\R}{\mathbb{R}}

\newcommand{\N}{\mathbb{N}}

\newcommand{\F}{\mathbb{F}}

\newcommand{\U}{\mathbb{U}}

\newcommand{\Lb}{\mathbf{L}}

\newcommand{\Xb}{\mathbf{X}}

\newcommand{\al}{\alpha}
\newcommand{\Ga}{\Gamma}
\newcommand{\ga}{\gamma}
\newcommand{\del}{\delta}

\newcommand{\ep}{\epsilon}
\newcommand{\sig}{\sigma}

\newcommand{\la}{\lambda}
\newcommand{\La}{\Lambda}
\newcommand{\tet}{\theta}
\newcommand{\Tet}{\Theta}
\newcommand{\om}{\omega}
\newcommand{\Om}{\Omega}

\newcommand{\br}{\vspace{3 mm}}
\newcommand{\imp}{\Rightarrow}
\newcommand{\nor}{\vartriangleleft}

\newcommand{\tri}{\bigtriangleup}

\newcommand{\inte}{{\rm{int\,}}}
\newcommand{\cls}{{\rm{cls\,}}}

\newcommand{\supp}{{\rm{supp\,}}}

\swapnumbers
\theoremstyle{plain}
\newtheorem{thm}{Theorem}[section]
\newtheorem{cor}[thm]{Corollary}

\newtheorem{prop}[thm]{Proposition}

\theoremstyle{definition}
\newtheorem{defn}[thm]{Definition}

\newtheorem{rmk}[thm]{Remark}

\newtheorem{exa}[thm]{Example}

\newtheorem{prob}[thm]{Problem}

\begin{document}

\title[Uniformly recurrent subgroups]
{Uniformly recurrent subgroups}

\date{February 20, 2014}

\author{Eli Glasner and Benjamin Weiss}

\address{Department of Mathematics\\
     Tel Aviv University\\
         Tel Aviv\\
         Israel}
\email{glasner@math.tau.ac.il}

\address {Institute of Mathematics\\
 Hebrew University of Jerusalem\\
Jerusalem\\
 Israel}
\email{weiss@math.huji.ac.il}

\begin{abstract}
We define the notion of uniformly recurrent subgroup, URS in short,
which is a topological analog of the notion of invariant random subgroup (IRS),
introduced in \cite{AGV}. Our main results are as follows. 
(i) It was shown in \cite{W} that for an arbitrary countable infinite group $G$,
any free ergodic probability measure preserving $G$-system admits a minimal model.
In contrast we show here, using URS's, that for the lamplighter group there is an ergodic
measure preserving action which does not admit a minimal model.
(ii) For an arbitrary countable group $G$, every URS can be realized as 
the stability system of some topologically transitive $G$-system.
\end{abstract}



%\thanks{{Both author's work is supported by ISF grant 
%\#1157/08}}

\maketitle

\tableofcontents

\section*{Introduction}

Let $G$ be a locally compact second countable topological group.
A {\em $G$-dynamical system} is a pair $(X,G)$ where $X$ is a compact metric
space and $G$ acts on $X$ by homeomorphisms. 
Given a compact dynamical system $(X,G)$, for 
$x \in X$ let $G_x =\{g \in G : gx =x\}$ be the {\em stability group at $x$}. 

Let $\mathcal{S}=\mathcal{S}(G)$ be the compact metrizable space of all subgroups of $G$ 
equipped with the Fell (or Chabauty) topology.
Recall that given a Hausdorff topological space $X$, a basis for the {\it Fell topology} on the hyperspace $2^X$, comprising the closed subsets of $X$, is given by the collection of sets $\{\Ucal(U_1,\dots,U_n; C)\}$, where 
$$
\Ucal(U_1,\dots,U_n; C) =\{A \in 2^X :  \forall \ 1 \le j \le n,\  A \cap U_j 
\not=\emptyset \ \& \ A \cap C = \emptyset\}.
$$
Here $\{U_1,\dots, \U_n\}$ ranges over finite collections of open subsets of 
$X$ and $C$ runs over the compact subsets of $X$.
The Fell topology is always compact and it is Hausdorff iff $X$ is locally compact (see e.g. \cite{Be}).
We let $G$ act on $\mathcal{S}(G)$ by conjugation.
This action makes $(\mathcal{S}(G),G)$ a $G$-dynamical system.
In order to avoid confusion we denote this action by $(g,H) \mapsto g\cdot H$
($g \in G, H \in \mathcal{S}(G)$).
Thus for a subgroup $H < G$ we have $g\cdot H = H^g = gHg^{-1}$.

Perhaps the first systematic study of the space $\mathcal{S}(G)$ is to be found in 
Auslander and Moore's memoir \cite{AM}. It then played a central role in the seminal work of
Stuck and Zimmer \cite{SZ}. 
More recently the notion of IRS (invariant random subgroup)
was introduced in the work of M. Abert, Y. Glasner and B. Virag \cite{AGV}.
Formally this object is just a $G$-invariant probability measure on $\mathcal{S}(G)$.
This latter work served as a catalyst and lead to a renewed vigorous interest in the study of IRS's
(see, among others, \cite{A-S}, \cite{AGV2},  \cite{B}, \cite{B2}, \cite{BGK}, and \cite{V}). 
A brief historical discussion of the subject can be found in \cite{AGV}.

Pursuing the well studied analogies between ergodic theory and topological dynamics
(see \cite{GW}) we propose to introduce a topological dynamical analogue of the notion of an IRS.

\begin{defn}\label{URS-def}
A minimal subsystem of $\mathcal{S}(G)$ is called a {\em uniformly recurrent subgroup},
URS in short. (Recall that according to Furstenberg a point $x$ in a compact dynamical system
$(X,G)$ is {\em uniformly recurrent} (i.e. for every neighborhood $U$ of $x$ the set $N(x,U)
=\{g \in G : gx \in U\}$ is syndetic) if and only if the orbit closure ${\cls}\{gx : g \in G\}$
is a minimal set.)
A topologically transitive subsystem of $\mathcal{S}(G)$ is called a {\em topologically transitive subgroup}, TTS in short.
\end{defn}

For later use we also define a notion of nonsingular random subgroup.

\begin{defn}
Recall that a {\em nonsingular action of $G$} is a measurable action of $G$ on a standard Lebesgue probability space $(X,\Bcal,\mu)$, where the action preserves the measure class of $\mu$ (i.e. $\mu(A) = 0 \iff \mu(gA) =0$ for every $A \in \Bcal$ and every $g \in G$).
We will call a nonsingular measure on $\mathcal{S}(G)$ a
{\em nonsingular random subgroup}, NSRS in short.
\end{defn}

In Section \ref{Sec1} we define and study the stability system which is associated to
a dynamical system and then consider some examples of groups possessing only trivial URS's. 
In Section \ref{Sec2} we examine homogeneity properties of URS's. In Section \ref{Sec3}
we show how a great variety of URS's can arise for lamplighter groups. In Section \ref{Sec4} 
we obtain some applications of URS's to ergodic theory. In Section \ref{Sec5} the richness of
the space $\mathcal{S}(\F_2)$ is demonstrated. Finally, in Section \ref{Sec6} we consider
the question of realization of URS's as stability systems.

We thank N. Avni, P.-E. Caprace, Y. Glasner and S. Mozes for their helpful advice.

\section{Stability systems}\label{Sec1}

It is easy to check that, whenever $(X,G)$ is a dynamical system, 
the map $\phi : X \to \mathcal{S}(G),\ x \mapsto G_x$ is upper-semi-continuous; i.e. $x_i \to x$ implies $\limsup \phi(x_i) \subset \phi(x)$.
In fact, if $x_i \to x$ and $g_i \to g$ (when $G$ is discrete the latter just means that 
eventually $g_i=g$) are convergent sequences in $X$ and
$G$ respectively, with $g_i \in G_{x_i}$, then $x_i=g_i x_i \to gx$, hence
$gx=x$ so that $\limsup G_{x_i} \subset G_x$. 
Recall that whenever $\psi : X \to 2^Z$ is an upper-semi-continuous map,
where $X$ is compact metric and $Z$ is locally compact and second countable, there exists a dense 
$G_\del$ subset $X_0 \subset X$ where $\psi$ is continuous at each point
$x_0 \in X_0$ (see e.g. \cite[page 95, Theorem 1]{C}).

\begin{defn}
Let $\pi : (X,G) \to (Y,G)$ be a homomorphism of $G$-systems; i.e.
$\pi$ is a continuous, surjective map and $\pi(gx) = g \pi(x)$ for every $x \in X$ and $g \in G$.
We say that $\pi$ is an {\em almost one-to-one extension} if there is a dense $G_\del$
subset $X_0 \subset X$ such that $\pi^{-1}(\pi(x)) = \{x\}$ for every $x \in X_0$.
\end{defn}

\begin{prop}\label{stability}
Let $(X,G)$ be a compact system. 
Denote by $\phi : X \to \mathcal{S}(G)$ the upper-semi-continuous map $x \mapsto G_x$
and let $X_0 \subset X$ denote the dense $G_\del$ subset of continuity points of $\phi$. 
Construct the diagram 
\begin{equation*}
\xymatrix
{
& \tilde{X} = X \vee Z \ar[dl]_\eta\ar[dr]^\al  & \\
X & & Z 
}
\end{equation*}
where 
\begin{gather*}
Z = {\cls} \{G_{x} : x \in X_0\}\subset \mathcal{S}(G),\\
\tilde{X} = {\cls}\{(x,G_{x}) : x \in X_0\} \subset X \times Z,
\end{gather*}
and $\eta$ and $\al$ are the restrictions to $\tilde{X}$ of the projection maps.
We have:
\begin{enumerate}
\item
The map $\eta$ is an almost one-to-one extension.
\end{enumerate}
If moreover $(X,G)$ is minimal then
\begin{enumerate}
\item[(2)] $Z$ and $\tilde{X}$ are minimal systems.
\item[(3)]
$Z$ is the unique minimal subset of the set ${\cls} \{G_x : x \in X\}\subset \mathcal{S}(G)$
and $\tilde{X}$ is the unique minimal subset of the set 
${\cls} \{(x,G_x) : x \in X\}\subset X \times \mathcal{S}(G)$.
\end{enumerate}
\end{prop}

\begin{proof}
(1)\ 
Let $\tilde{X}_0 =\{(x,G_x) : x \in X_0\}$. It is easy to see that the fact that $x \in X_0$
implies that the fiber $\eta^{-1}(\eta((x,G_x)))$ is the singleton $\{(x,G_x)\}$. 

(2)\ 
Fix a point $x_0 \in X_0$. The minimality of $(X,G)$ implies that the orbit of the point
$(x_0,G_{x_0})$ is dense in $\tilde{X}$. On the other hand,
if $(x,L)$ is an arbitrary point in $\tilde{X}$ then, again by minimality of $(X,G)$, there 
is a sequence $g_n \in G$ with $\lim_{n}g_nx = x_0$. We can assume that
the limit $\lim_n g_n (x,L) = (x_0,K)$ exists as well, and then the fact that $x_0 \in X_0$
implies that $K = G_{x_0}$.
This shows that $(\tilde{X},G)$ is minimal and then so is $Z = \al(\tilde{X})$.

(3)\ 
Given any $x\in X$ we argue, as in part (2), that $(x_0,G_{x_0})$ is in the orbit closure
of $(x,G_x)$.
\end{proof}

\begin{defn}\label{stab-sys-def}
Given a dynamical system $(X,G)$ we call the system $Z \subset \mathcal{S}(G)$
{\em the stability system of $(X,G)$}. We denote it by $Z = \mathcal{S}_X$. 
We say that $(X,G)$ is {\em essentially free} if $Z = \{e\}$. 
Note that when $Z = \{N\}$ is a singleton then the subgroup $N < G$ is necessarily a normal
subgroup of $G$ and, by the upper-semi-continuity of the map $x \mapsto G_x$, it follows that
$N < G_x$ for every $x \in X$, whence $N = \bigcap \{G_x : x \in X\}$. In this case then, the action reduces to an action of the group $G/N$ and the latter is essentially free. 
In particular, if the action of $G$ on $X$ is {\em effective}
(i.e. $\forall g \not = e, \exists x \in X, gx \not=x$)
then $Z =\{N\}$ implies that $N = \{e\}$, so that the action is essentially free.
Also note that part (3) of the proposition implies that if $(X,G)$ is minimal and there is some point 
$x \in X$ with $G_x =\{e\}$ then $(X,G)$ is necessarily essentially free.
\end{defn}

\begin{prop}\label{Z}
Let $X \subset \mathcal{S}(G)$ be a URS and let $X_0 \subset X$ denote the dense $G_\del$ subset 
of $X$ consisting of the continuity points of the map $x \mapsto G_x = N_G(x)$. 
Consider 
$$
Z = \mathcal{S}_X = \cls \{G_x : x \in X_0\} = \cls \{N_G(x) : x \in X_0\},
$$
the stability system of $X$.
\begin{enumerate}
\item 
$Z$ is again a URS.
\item
If $N_G(x_0) = x_0$ for some $x_0 \in X_0$ then $Z=X$.
\item
Conversely, if $Z = X$ then $N_G(x) =x$ for every $x \in X_0$.
\item
If for some $x_0 \in X_0$ we have that $N = N_G(x_0)$ is a co-compact subgroup of $G$, 
then $Z =\mathcal{S}_X$ is a factor of the homogeneous $G$-space $G/N$.
\item
If, in addition, $N_G(N) = N$, then $Z \cong G/N$.
\end{enumerate}
\end{prop}

\begin{proof}
1. \ The first part is a direct consequence of Proposition \ref{stability}, and we also deduce that
$Z$ is a factor of an almost 1-1 extension of $X$, namely of $\tilde{X}$.

2. \ The second follows since by our assumption $x \in X \cap Z$.

3. \ Now suppose $Z=X$ and let $x_0$ be a point in $X_0$. We have $N =G_{x_0}=N_G(x_0) \in Z$ 
and thus $N \in X$. 
Let now $X_{max} = \{x \in X : y \supset x \ \& \ y \in X \imp y =x\}$. As in \cite[Lemma 5.3]{G-90}
one shows that the set $X_{max}$ is a dense $G_\del$ subset of $X$. We now further assume that 
$x_0 \in X_0 \cap X_{max}$. 
But then the inclusion $x_0\subset N$ implies $x_0=N$ and we have
$$
(x_0, N) = (x_0, G_{x_0}) = (x_0,x_0) \in \tilde{X}.
$$
This implies $\tilde{X} = \{(x,x) : x \in X\}$ and, in particular 
$$
(x,G_x) = (x,N_G(x)) = (x,x),
$$
for every $x \in X_0$.

4.\ For the fourth part note that, by assumption, $KN =G$ for some compact subset 
$K \subset G$. It follows that $G \cdot N = K \cdot N = \{k N k^{-1} : k \in K\}$, the $G$-orbit of $N$, is compact and therefore $Z =G \cdot N$. 
Moreover, the map $\psi : G/N \to Z$ defined by $\psi(gN) = gNg^{-1}$ is a homomorphism.
Finally $\psi$ is one-to-one when $N_G(N) = N$, whence follows part (5). 
\end{proof}

The proof of the next proposition is straightforward.

\begin{prop}\label{gen}
\begin{enumerate}
\item
A surjective group homomorphism $\eta : A \to B$ between two countable groups 
$A$ and $B$ induces an embedding $\eta_* : \mathcal{S}(B)  \to \mathcal{S}(A)$ of the corresponding dynamical systems. Explicitly, for $H < B$ its image in $\mathcal{S}(A)$
is the subgroup $\tilde{H} = \eta^{-1}(H)$. Moreover we have
$$
\tilde{H}^a = \eta_*(H)^a = \eta_*(H^{\eta(a)})=\widetilde{H^{\eta(a)}}
$$
for every $a \in A$.
\item
Let $G$ be a group and $H < G$ a subgroup of finite index in $G$. If $Z_0 \subset \mathcal{S}(H)$ 
is a $H$-URS, then $Z = \cup \{gZ_0g^{-1} : g \in G\}$ is a $G$-URS in $\mathcal{S}(G)$. 
\end{enumerate}
\end{prop}

For further information on the space $\mathcal{S}(G)$ see \cite{Sch}.

\br

Clearly for an abelian group $G$ the conjugation action on $\mathcal{S}(G)$
is the identity action. If $G$ is a finitely generated nilpotent group then every
subgroup of $G$ is finitely generated as well \cite{Ba}. Thus $G$ has only countably many
subgroups and we conclude that every IRS (hence also every URS) of $G$ is finite 
\footnote{We thank Yair Glasner for this observation and for suggesting example \ref{mon} below.}.

Let $G$ be a connected semisimple Lie group with finite center and $\R$-rank $\ge 2$,
satisfying  property (T) (e.g. $G = SL(3,\R)$).
It follows from the Stuck-Zimmer theorem \cite{SZ} that any URS in $\mathcal{S}(G)$,
which supports a $G$-invariant probability measure, is necessarily of the form 
$\{g\Ga g^{-1} : g \in G\} \cong G/N_G(\Gamma)
\cong \mathcal{S}_{G/\Ga}$,
where $\Gamma < G$ is a co-compact lattice and $N_G(\Gamma)
=\{g \in G : g\Ga g^{-1} =\Ga\}$ its normalizer in $G$. Moreover, for each
parabolic subgroup $Q < G$, the homogeneous space $G/Q$, with left multiplication, forms
a minimal $G$ action and clearly the corresponding stability system $\mathcal{S}_{G/Q} =
\{gQg^{-1} : g \in G\} \cong G/Q$ is a URS. We don't know whether, up to isomorphism, these  are the onlyURS's in $\mathcal{S}(G)$.

\br

In this connection we have the following theorems of Stuck \cite[Theorem 3.1 and Corollary 3.2]{S}.
(The action of a topological group $G$ on a compact Hausdorff space $X$ is {\em locally free} if 
for every $x \in X$ the stability subgroup $G_x$ is discrete.)

\begin{thm}[Stuck]\label{Stuck1}
Let $G$ be a real algebraic group acting minimally on a compact Hausdorff space $X$. 
Fix $x \in X$ and let $N =N_G((G_x)^0)$.  Then: 
\begin{enumerate}
\item
$N$ is co-compact in $G$. 
\item 
If $H$ is any closed subgroup of $G$ containing $N$ then there is a compact minimal $H$-space $Y$ such that $X$ is $G$-equivariantly homeomorphic to $G \times_H Y$. 
\end{enumerate}
\end{thm}

\begin{thm}[Stuck]\label{Stuck2}
Let $G$ be a semisimple Lie group with finite center and without compact factors, acting minimally on a compact Hausdorff space $X$. Then either the action is locally free, or it is isomorphic to an induced action 
$G\times_Q Y$, where $Q$ is a proper parabolic subgroup of $G$ and $Y$ is a compact $Q$-minimal space.
\end{thm}

\begin{prob}\label{pro1}\footnote{See also Problem \ref{pro2} below.}
For a semisimple Lie group $G$ as above, is it the case that every nontrivial URS of $G$ is either 
of the form $\{g\Ga g^{-1} : g \in G\}$,
where $\Gamma < G$ is a co-compact lattice, or it admits $G/Q$ as a factor
with $Q < G$ a proper parabolic subgroup of $G$?
\end{prob}

We are currently working on that problem and have some indications that an affirmative answer is plausible.
See \cite{S} for information on minimal actions of semisimple Lie groups.

\br

We next show that certain non-abelian infinite countable 
groups admit no nontrivial URS's. On the other hand, in Section \ref{Sec3} and \ref{Sec5} we 
will see many examples of nontrivial URS's.

\begin{exa}\label{mon}
A ``Tarski monster" group $G$ is a countable noncyclic group with the property that its only
proper subgroups are cyclic (either all of a fixed prime order $p$, or all infinite cyclic).
It is easy to see that such a group is necessarily simple. 
Since $G$ is countable it follows that $\mathcal{S}(G)$
is a countable set. Moreover, the only URS's in $\mathcal{S}(G)$ are $\{e\}$ and $\{G\}$.
In fact, if $X \subset \mathcal{S}(G)$ is an URS then, being a countable space, it must
have an isolated point $H \in X$. Since $X$ is minimal it must be finite (finitely many 
conjugates of the open set $\{H\}$ must cover $X$). If $H \in X$ and $H$ is neither
$\{e\}$ nor $G$, then $H$ is a cyclic group and $X = \{gHg^{-1} : g \in G\}$ is a
finite set. Now $G$ acts (by conjugation) on the finite set $X$ as a group of permutations and the
kernel of the homomorphism from $G$ onto this group of permutations is a
normal subgroup of $G$. As $G$ is simple this kernel is either $\{e\}$ or $G$ and 
both cases lead to contradictions.
\end{exa}

\begin{exa}
It is not hard to see that $\{e\}$ and $\{G\}$ are the only URS's for $G=S_\infty(\N)$,
the countable group of finitely supported permutations on $\N$.
This paucity of URS's is in sharp contrast to the abundance of IRS's of this group as
described in \cite{V}.
\end{exa}

\br

\section{Homogeneity properties of a URS}\label{Sec2}
In this section we let $G$ be a countable discrete group.

\begin{defn}
We say that a property P of groups is {\em admissible} if
\begin{enumerate}
\item
P is preserved under isomorphisms.
\item
P is inherited by subgroups, i.e. if $H$ has P and $K < H$ then $K$ has P.
\item
P is preserved under increasing unions, i.e. if $H_n \nearrow H$, where $\{H_n\}_{n \in \N}$
is an increasing sequence of subgroups of $H$ and each $H_n$ has P, then so does $H$.
\end{enumerate}
\end{defn}

\begin{prop}\label{closed}
Let P be an admissible property of groups. Then the subset
$$
\mathcal{S}(G)_P =
\{H < G : H \ {\text{has P}}\}
$$
is a closed invariant subset of the dynamical system $\mathcal{S}(G)$.
\end{prop}

\begin{proof}
Suppose $H_n \to H$ in $\mathcal{S}(G)$ with $H_n \in \mathcal{S}(G)_P$
for every $n \in \N$. 
Let $H = \{h_0=e, h_1, h_2,\dots\}$ be an enumeration of $H$. Given $k \ge 0$ there
exits $n_0$ such that for $n \ge n_0$ we have $\{h_0=e, h_1, h_2,\dots,h_k\}\subset H_n$.
Let $H^{(k)}$ be the subgroup of $H$ (and of $H_n$) generated by the set $\{h_0=e, h_1, h_2,\dots,h_k\}$.
It is now clear that each $H^{(k)}$ has P and that $H^{(k)} \nearrow H$. It thus follows that
also $H$ has P.
\end{proof}

The assertions in the next proposition are well known and not hard to check. 
\begin{prop}
The following properties are admissible:
\begin{enumerate}
\item
Commutativity.
\item
Nilpotency of degree $\le d$.
\item
Solvability of degree $\le d$
\item
Having an exponent $q \ge 2$ (i.e. satisfying the identity $x^q =e$).
\item
Amenability.
\end{enumerate}
\end{prop}

\begin{rmk}
The question whether, for a general locally compact topological group $G$, the collection of closed 
amenable subgroups of $G$ forms a closed subset of $\mathcal{S}(G)$ is open.
As is shown in \cite{Ca-Mo} this is true for a very large class of groups (which includes 
the discrete groups). The discrete group case follows directly from Schochetman's work \cite{Sch}. 
\end{rmk}

\begin{cor}
Let $Z \subset \mathcal{S}(G)$ be a URS and let P be an admissible property. Then
either every element of $Z$ has P, or none has P. 
\end{cor}

\begin{rmk}
One can easily check that e.g. nilpotency and being perfect (i.e. the property:
$[H,H] =H$) are not admissible properties.
\end{rmk}

We next consider topologically transitive subgroups (TTS) (i.e.
closed invariant topologically transitive subsets of $\mathcal{S}(G)$).
In the spirit of \cite{GK} let us say that a subset $\mathcal{L} \subset \mathcal{S}(G)$
is a {\em dynamical property} if it is Baire measurable, and $G$-invariant;
i.e. invariant under conjugations.
In view of Proposition \ref{closed} every admissible property of groups defines
a corresponding dynamical property in $\mathcal{S}(G)$.
The next proposition is just a special instance of the general ``zero-one law" for
topologically transitive dynamical systems, see e.g. \cite{GK}.

\begin{prop}\label{TTS}
Let $Z  \subset \mathcal{S}(G)$ be a TTS and $\mathcal{L} \subset \mathcal{S}(G)$ be
a dynamical property. Then the set $Z \cap \mathcal{L}$ is either meager or comeager.
\end{prop}

\begin{cor}
Let $Z  \subset \mathcal{S}(G)$ be a URS then the set of perfect elements in $Z$
is either meager or comeager.
\end{cor}

For the corresponding measure theory zero-one law 
see Proposition \ref{m01} below. 
\br

\section{URS's for Wreath products}\label{Ex}\label{Sec3}

Let $\Ga$ be an arbitrary countable infinite group and let $G = \Z_2 \wr \Ga = \{1,-1\} \wr \Ga$ 
be the corresponding (restricted) Wreath product. Recall that the group $\Z_2 \wr \Ga$ is defined as the
semidirect product $G = K \rtimes \Ga$, where $K = \oplus_{\ga \in \Ga} \Z_2$ 
and $\Ga$ acts on $K$ by permutations: $(\ga\cdot k)(\ga')= k(\ga'\ga), \ \ga, \ga' \in \Ga, k \in K$. 
Thus, for $(l,\ga), (k,\ga') \in K \wr \Ga$,
$$
(l,\ga)(k,\ga') = (l (\ga \cdot k), \ga\ga').
$$

Next let $\Om = \{-1,0,1\}^\Ga$ be the product space equipped with the compact product topology.
We let $G$ act on $\Om$ as follows. The elements of $K$ act by coordinatewise left multiplication
(in the semigroup $ \{-1,0,1\}$), while the action of $\ga \in \Ga$ is again via the corresponding permutation.
Thus, for $\om \in \Om, k \in K, \ga,\ga' \in \Ga$,
$$
((k,\ga)\om)(\ga') = k(\ga)\om(\ga'\ga).
$$ 

Consider the dynamical system $(\Om,G)$ and, as above, let $G_\om =\{g \in G : g\om =\om\}$.
It is easy to see that the set $\{G_\om : \om \in \Om\}$ is a closed invariant subset of 
$\mathcal{S}(G)$ which is isomorphic to the $\Ga$ $2$-shift $(\Theta,G): =(\{0,1\}^\Ga, G)$ (where $K$ acts trivially).
In fact for a point $\om \in \Om$ the corresponding subgroup $G_\om$ is the subgroup 
$\oplus_A \{1,-1\}$, where $A=A(\om) = \{\ga \in \Ga : \om(\ga)=0\}$. Thus the groups 
$\{G_\om : \om \in \Om\}$ are in one-to-one correspondence with the functions 
$\{\ch_A : A \subset \Ga\} = \{0,1\}^\Ga$. Moreover, we see that the map 
$\phi : \om  \to G_\om$ can be viewed as the
homomorphism $\om \to \ch_{A(\om)}$, $\Om \to \{0,1\}^\Ga$.
This observation shows that the dynamical system $\mathcal{S}_\Om =
\{G_\om : \om \in \Om\} \subset \mathcal{S}(G)$, the stability system of $(\Om,G)$, is isomorphic
to the dynamical system $(\Theta,G) = (\{0,1\}^\Ga,G) = (\{0,1\}^\Ga,\Ga)$.

Now to every $\Ga$-invariant closed subset $\La \subset \Theta$ we associate the 
$G$-system $\tilde\La := \phi^{-1}(\La)\subset \Om$. 
If $\La$ is a $\Ga$-minimal system, then
$$
\tilde{\La} = \{\om \in \Om : \ch_{A(\om)} \in \La\},
$$
is a $G$ minimal system (due to the fact that 
$K= \oplus_{\ga \in \Ga} \{1,-1\}$ densely embeds as a subgroup of the compact group $\otimes_{\ga \in \Ga}\{1,-1\}$)
and its stability system $\mathcal{S}_{\tilde\La} =
\{G_\om : \om \in \tilde{\La}\}$ is a $G$-URS which is isomorphic to $\La$. 
Thus we have shown that {\em every $\Ga$-minimal subset of the $2$-shift $(\{0,1\}^\Ga,\Ga)$
appears as a stability system, hence as URS in $\mathcal{S}(G)$}.

For some background on the notion of RIM (relatively invariant measure) see \cite{Gl-RIM}. For the notions of 
proximality and (topological) weak mixing see e.g. \cite{Gl-prox}.
The proof of the next proposition is straightforward.

\begin{prop}\label{RIM}
For $\theta \in \Theta=\{0,1\}^\Ga$ let $\la_\tet$ denote the Haar measure on the compact group $\{1,-1\}^{A^c(\tet)}$, 
where $A(\tet) = \{\ga \in \Ga : \tet(\ga)=1\}$.
Then the extension $\phi : \Om \to \Tet$ is a measure preserving homomorphism in the sense that 
the {\em section} \  $\la : \Tet \to M(\Om)$, with $\supp(\la_\tet) = \phi^{-1}(\tet)$,
naturally defined by the family $\{\la_\tet : \tet \in \Tet\}$, is a RIM.
\end{prop}

\begin{prop}\label{lift}
For $\La$ a minimal subsystem of $(\{0,1\}^\Ga,\Ga)$ let $\phi_\La : \tilde\La \to \La$ be 
the restriction of $\phi$ to $\tilde\La$. Then
\begin{enumerate}
\item
$\phi_\La$ is an open  homomorphism of $G$-systems.
\item
For every $\om \in \tilde\La$ and $k \in K$ the pair $(\om, k\om) \in R_\phi^{(2)}$ is $G$-proximal.  
\item
For every $n \ge 1$ the relative $n$-proximal relation $P_\phi^{(n)}$ is dense in $R_\phi^{(n)}$.
\item
The extension $\phi_\La$ is weakly mixing. 
\item
For $\tet \in \La$ let $\la_\tet$ denote the Haar measure on the compact group $\{1,-1\}^{A^c(\tet)}$, 
then the section $\tet \mapsto \la_\tet$ is a  RIM for the extension $\phi_\La : \tilde\La \to \La$.
Thus, in particular, $\la$ defines an injection of the set of $\Ga$-nonsingular probability measures
on $\La$ into the set of NSRS's of $G$ (the $G$-nonsingular probability measures on
$\tilde{\La}$) and moreover a $\Ga$-invariant measure on $\La$ is lifted to a $G$-invariant measure on 
$\tilde{\La}$.
\end{enumerate}
\end{prop}

\begin{proof}
Claim (1) is easy to check. Claim (2) follows from the fact that $\om$ and $k\om$ differ in only finitely
many coordinates. Claim (3) is a consequence of the fact that the set 
$K\times K \times \cdots \times K (\om, \om,\dots,\om)$ ($n$-times) is dense in 
$\phi^{-1}(\phi(\om)) \times \phi^{-1}(\phi(\om)) \times \cdots \times \phi^{-1}(\phi(\om))$.
Claim (4) is now a consequence of an old result of Glasner and van der Woude 
(see \cite[Theorem 6.3]{G-Bohr})
which ensures that an open homomorphism between minimal systems which satisfies condition 
(3) is indeed weakly mixing.
Finally the details of claim (5) are easily verified.
\end{proof}

\begin{exa}
Let $\Ga$ be the free group on two
generators, $\Ga=\mathbb{F}_2=\langle a,b \rangle$.
Let $Z$ be the space of right infinite reduced words on the letters
$\{a,a^{-1},b,b^{-1}\}$. Then $\Ga$ acts on $Z$ by concatenation on the left
and cancelation.

Let $C(a) = \{z \in Z :\  \text{the first letter of $z$ is $a$}  \}$ and let 
$\psi : Z \to \{0,1\}^\Ga$ be defined by $\psi(z)(\ga) = \ch_{C(a)}(\ga z),\ z \in Z, \ga \in \Ga$.
It is not hard to check that $\psi$ is an isomorphism of the minimal system $(Z,\Ga)$ into
the system $(\{0,1\}^\Ga,\Ga)$. Let $\La = \psi(Z)$, then $\La$ is a minimal
subsystem of $(\{0,1\}^\Ga,\Ga)$, and thus, via the construction described above, we 
obtain a realization of the system $(Z,\Ga)$ as a URS of the group $G = \Z_2 \wr \Ga$,
namely as the stability system of the dynamical system $(\tilde{\La}, G)$.

Let $m$ be the probability measure $m=\frac 14(\del_a+\del_b+\del_{a^{-1}}+\del_{b^{-1}})$ on $\Ga$
and let  $\eta$ be the probability measure on $Z$ given by
$$
\eta(C(\ep_1,\dots,\ep_n))=\frac1{4\cdot 3^{n-1}},
$$
where for $\ep_j\in\{a,a^{-1},b,b^{-1}\}$,\
$C(\ep_1,\dots,\ep_n)=\{z\in Z:z_j=\ep_j,\ j=1,\dots,n\}$. 
The measure $\eta$ is $m$-stationary (i.e. $m *\eta= \eta$) and the $m$-system 
${\bf{Z}} = (Z,\eta,\Ga)$ is the Poisson boundary $\Pi(\mathbb{F}_2,m)$ (see e.g. \cite{FG}). 
The push forward  $\psi_*(\eta)$ is an $m$-stationary measure on $\La$.
The fact that the system $(Z,\Ga)$ is strongly proximal (see \cite{Gl-prox}) shows that there is no $\Ga$ 
invariant measure on the URS $(\La,\Ga)$. Thus we have the following:

\begin{prop}
The URS $\mathcal{S}_{\tilde\La}\cong (\La,\Ga)$ carries the probability measure $\eta$
which is an ergodic NSRS, but it admits no IRS.
\end{prop}
\end{exa}

\br

\section{Applications to ergodic systems}\label{Sec4}

Our first application is a direct consequence of Proposition \ref{TTS}.

\begin{prop}\label{m01}
Let $\Xb = (X,\mathcal{B},\mu,G)$ be an ergodic nonsingular dynamical system.
Let $\phi: X \to \mathcal{S}(\Ga)$ be the map $x \mapsto G_x = \{g \in G : g x =x\}$.
Let $\nu = \phi_*(\mu)$, the push forward probability measure on $\mathcal{S}(G)$.
Finally, let $Z = \supp(\nu)$. Then $Z$ is a TTS and for 
every dynamical property $\mathcal{L} \subset \mathcal{S}(\Ga)$ 
the set $Z \cap \mathcal{L}$ is either meager or comeager.
If moreover the dynamical property $\mathcal{L}$ is $\nu$-measurable
then $\nu(Z \cap \mathcal{L})$ is either $0$ or $1$.
\end{prop}

\begin{proof}
The ergodicity of $\Xb$ implies the ergodicity of the factor nonsingular system $(Z,\nu,G)$.
Now the latter is a topological system, where by construction $\supp(\nu)=Z$.
In this situation ergodicity implies topological transitivity and the zero-one law applies.
The last assertion follows directly from the ergodicity of $\Xb$.
\end{proof}

Thus e.g. let us single out the following.

\begin{prop}
For an ergodic nonsingular $\Xb$ either $\mu$-a.e. $G_x$ is amenable
or $\mu$-a.e. $G_x$ is non amenable.
\end{prop}

For more results on the stability systems of ergodic actions of Lie groups we refer to \cite{Go-S}.

\br

For our next application consider a  
dynamical system $(X,G)$ (either compact or Borel). For each $g \in G$
let $F_g = \{x \in X : gx=x\}$. We then have, for every $x \in X$,
$$
G_x = \{g \in G : x \in F_g\}.
$$

\begin{cor}\label{measure}
Let $(\Om,\mu,G)$ be a probability measure preserving system and
suppose it admits a minimal model $(X,\nu,G)$. Thus, we assume the existence
of a measure isomorphism $\rho : (\Om,\mu,G) \to (X,\nu,G)$. 
Then for $\mu$-a.e. $\om \in \Om$ the orbit closure 
$\cls \{g G_\om g^{-1} : g \in G\}$ must contain $\mathcal{S}_X$ as a unique minimal subset.
\end{cor}

\begin{proof}
For each $g \in G$ let
$$
A_g = \{\om \in \Om : g\om = \om\}.
$$
Then for every $\om \in \Om$ we have
$$
G_\om = \{g \in G : \om \in A_g\}.
$$
Now for every $g \in G$ we have $\nu(\rho(A_g) \tri F_g) =0$ and it follows that
for $\mu$-a.e. $\om$, \ $G_\om = G_{\rho(\om)}$.
Our claim now follows from  Proposition \ref{stability}.(3).
\end{proof}

\begin{exa}\label{lamp-ex}
As in Section \ref{Ex}, 
let $\Om= \{-1,0,1\}^\Z$ and $\sig : \Om \to \Om$ the shift map. Let $\mu$ denote the
product measure $\{1/4,1/2,1/4\}^\Z$ on $\Om$. Set $\Lb_0 = \oplus_\Z \{-1,1\}$ 
(the countable direct sum) and let $\Lb_0$ act on $\Om$ by coordinate-wise multiplication. 
Let $\Lb = \{1,-1\} \wr \Z$, the {\em lamplighter group}. We view $\Lb$
as the group of homeomorphisms of $\Om$ generated by $\Lb_0$ and $\sig$.
As in Section \ref{Ex} we observe that the set $\{\Lb_\om : \om \in \Om\}$ is a closed invariant subset of 
$\mathcal{S}(\Lb)$ which is isomorphic to the $2$-shift $(\{0,1\}^\Z, \sig)$ (where $\Lb_0$ acts trivially).
For a point $\om \in \Om$ the corresponding subgroup $\Lb_\om$ is the subgroup 
$\oplus_A \{-1,1\}$, where $A=A(\om) = \{n \in \Z : \om(n)=0\}$. The groups 
$\{\Lb_\om : \om \in \Om\}$ are in one-to-one correspondence with the functions 
$\{\ch_A : A \subset \Z\} = \{0,1\}^\Z$, and the map 
$\phi : \om  \to \Lb_\om$ is a homomorphism $\om \mapsto \ch_{A_\om}$, $\Om \to \{0,1\}^\Z$.
Thus the dynamical system $\mathcal{S}_\Om =
\{\Lb_\om : \om \in \Om\} \subset \mathcal{S}(\Lb)$, the stability system of $(\Om,\Lb)$, is isomorphic
to the full shift dynamical system $(\{0,1\}^\Z,\sig)$.
 \end{exa}
 
It was shown in \cite{W} that for an arbitrary countable infinite group $G$,
any {\bf free} ergodic probability measure preserving $G$-system admits a minimal model.
In contrast we show next that for the lamplighter group there is an ergodic
measure preserving action which does not admit a minimal model.

\begin{thm}\label{lamp} 
For the lamplighter group $\Lb$ there is an ergodic dynamical system for which no
minimal model exits.
\end{thm}

\begin{proof}
Consider the dynamical system $(\Om,\mu,\Lb)$ presented in Example \ref{lamp-ex}.
Suppose it admits a minimal model $(X,\nu,\Lb)$. Thus, we assume the existence
of a measure isomorphism $\rho : (\Om,\mu,\Lb) \to (X,\nu,\Lb)$.
However, the topological system $(\{0,1\}^\Z,\Lb) = (\{0,1\}^\Z,\sig)$ contains uncountably
many distinct minimal sets and applying Corollary \ref{measure} we 
arrive at a contradiction.
\end{proof}

\begin{prop}
Let $(X,G)$ be a minimal system. Let $\mu$ be a $G$-invariant probability measure on $X$
such that the action of $G$ on the probability space $(X,\mu)$ is $\mu$ essentially free.
Then the action $(X,G)$ is essentially free.
\end{prop}

\begin{proof}
For each $g \in G$ we set $F_g = \{x \in X : gx=x\}$, a closed subset of $X$. If for some $g \not=e$,
$\inte F_g \not=\emptyset$ then, by minimality, $\mu(F_g) >0$ which contradicts our assumption
that the measure action is essentially free. Thus each set $F_g, g \not = e$ is nowhere dense, 
hence the set
$F = \bigcup \{F_g : e\not = g \in G\}$ is meager. 
Since $G_x = \{e\}$ for every $x \in X \setminus F$,
we conclude that the system $(X,G)$ is indeed essentially free.
\end{proof}

\begin{rmk}
The same proof works assuming only that $\mu$ is nonsingular.
\end{rmk}

\br

\section{Uncountably many URS's for $\mathbb{F}_2$}\label{Sec5}

\begin{thm}\label{many}
For the free group $G = \mathbb{F}_2$, the space $\mathcal{S}(G)$ contains uncountablly many non-isomorphic infinite URS's.
\end{thm}

\begin{proof}
Our construction is based on Example \ref{lamp} above and the following basic observation.
The lamplighter group $\Lb$ is generated by two elements.
Explicitly we can take these to be $e_1$ and $\sig$, where $e_1 \in \Lb_0$
is defined by $e_1(n) = (-1)^{\del_{1n}}$.
Let $\eta : \mathbb{F}_2 \to \Lb$ be the surjective group homomorphism which is determined by
$\eta(a) = e_1$ and $\eta(b) = \sig$. Next define an action of $\mathbb{F}_2$ on
$\Om$ by letting $g\om = \eta(g)\om,\ g \in \mathbb{F}_2, \om \in \Om$.
Clearly then $G_\om = \eta^{-1}(\Lb_\om),\ (\om \in \Om)$ and again we see that the 
dynamical system $\mathcal{S}_\Om =
\{G_\om : \om \in \Om\} \subset \mathcal{S}(G)$, the $G$-stability system of $(\Om,G)$, is isomorphic to the dynamical system $(\{0,1\}^\Z,\sig)$ (see Proposition \ref{gen}.(1)). 
As the latter system contains an uncountable 
family of pairwise non isomorphic minimal subsystems, our proof is complete.
\end{proof}

\begin{rmk}
This is of course a much stronger assertion than the claim in \cite[Lemma 3.9]{SZ} 
that the action of $\mathbb{F}_2$ on $\mathcal{S}(\mathbb{F}_2)$ is not tame.
\end{rmk}

\begin{rmk}
By construction these URS's are noneffective. A more sophisticated
construction due to Louis Bowen \cite[Theorem 5.1]{B} and the main result of
\cite{GJS} (actually \cite[Theorem 3.1]{GU} will suffice here) yield an uncountable family
of pairwise nonisomorphic effective URS's for $\mathbb{F}_2$.  
\end{rmk}

Since $\mathbb{F}_2$ sits as a finite index subgroup in the group in $SL(2,\Z)$ it follows that
this latter group also admits an uncountable family of pairwise non isomorphic URS's.
By the well known work of Stuck and Zimmer, for the group $G = SL(3,\Z)$, 
any IRS on $\mathcal{S}(G)$ is finite (see \cite[Corollaries 4.4 and 4.5]{SZ}).
We do not know whether the same holds for URS's on $\mathcal{S}(G)$.

\begin{prob}\label{pro2}
Are there infinite URS's in $\mathcal{S}(SL(3,\Z))$ ?
\end{prob}

\br

\section{Realizations of URS's as stability systems}\label{Sec6}

A basic result proved in \cite{AGV} is that to every ergodic IRS $\nu$ (a $G$-invariant
probability measure on $\mathcal{S}(G)$) there corresponds an 
ergodic probability measure preserving system $\Xb = (X,\mathcal{B},\mu,G)$ whose
stability system is $(\mathcal{S}(G),\nu)$. See the recent works of Creutz and Peterson 
\cite[Theorem 3.3]{Cr-P} and Creutz \cite[Theorem 3.3]{Cr} for continuous and 
NSRS versions of this theorem.
We expected to be able to obtain
an analogous statement for URS's. However, the question whether this desired analog holds remains open 
(see Problem \ref{min-stab} below) and we only have the following:

\begin{prop}\label{real}
Let $G$ be a countable infinite group.  
For every URS $Z \subset \mathcal{S}(G)$ there is a topologically transitive system $(X,G)$ with 
$Z$ as its stability system.
\end{prop}

\begin{proof}
We begin with the case where some $H \in Z$ has finite index in $G$.
It then follows that $Z$ is finite and we let $Y = G/H = \{gH : g\in G\}$, the finite homogeneous
$G$-space of right $H$-cosets.
We now take $X$ to be the $G$-orbit of the point $x_0=(H,H)$ in the product system
$(Z \times Y, G)$. Note that $G$ acts on $Z$ by conjugation, whereas it acts on $Y$
by multiplication on the left. The stability subgroup $G_{x_0}$ is $H$ and it follows that
indeed $Z$ is the stability system of the (minimal) finite system $(X,G)$.

We now assume that $[G:H] = \infty$ for every $H \in Z$.

Set $\Om = \{0,1\}^G$ and let $G$ act on $\Om$ by $g\om(h) = \om(g^{-1}h)$.
For $H < G$ let
$$
\Om_H = \{\om \in \Om : h\om = \om,\ \forall h \in H\},
$$
and
$$
\Theta = \bigcup \{\{H\} \times \Om_H : H \in Z\}.
$$
It is easy to check that $\Theta$ is a closed $G$-invariant subset of
$Z \times \Om$.
Also note that the map $\pi : (\Theta,G) \to (Z,G)$ is clearly a homomorphism of $G$-systems.  

1. \ 
Our first observation is that for each $x =(H,\om) \in \Theta$ we have $\pi(x)=H$ and
$$
H \subset G_x = \{g \in G : gx=x\} \subset G_{\pi(x)}= G_H = \{g\in G : gHg^{-1} =H\} = N_G(H),
$$
the normalizer of $H$ in $G$.
Also note that 
$H \nor N_G(H)$, so that $\tilde{H} =N_G(H)/H$ is a group.
Finally observe that for every $x \in \Theta$ the group $\tilde{H}$, with $H =\pi(x)$, acts on the fiber
$\pi^{-1}(H) = \{H\} \times \Om_H \cong \Om_H$.

2. \ 
Given $H \in Z$ consider the map $\rho_H : \Om_H \to \{0,1\}^{\tilde{H}}$ which sends an element $\om
\in \Om_H \subset \Om = \{0,1\}^G$ to its restriction to the subset $N_G(H) \subset G$.
(We identify  the image $\rho_H(\Om_H)$ with $\{0,1\}^{\tilde{H}}$.)
Clearly $\rho_H$ is a surjective homomorphism of $\tilde{H}$-systems.

3.\ 
We now fix some (arbitrary) point $H \in Z$.
We claim that there exists a $\tilde{H}$-minimal subsystem
of the symbolic system $(\{0,1\}^{\tilde{H}},\tilde{H})$, say $\tilde{Y}$, 
such that the action of $\tilde{H}$ on $\tilde{Y}$ is free.
There are two cases we need to consider. The first case is when the group $\tilde{H}$
is finite. We then take $\tilde{Y}$ to be $\tilde{H}\chi_e$ where $\chi_e \in \{0,1\}^{\tilde{H}}$
is the function $\chi_e(h) = \del_{e,h}$. 
In the second case, where $\tilde{H}$ is infinite,
we apply a recent result of Gao, Jackson and Seward \cite{GJS},
which ensures the existence of a free $\tilde{H}$-minimal subsystem $\tilde{Y}$
of $(\{0,1\}^{\tilde{H}},\tilde{H})$. 
In each of these cases choose $Y$
to be any minimal subsystem of the system $(\Om_H,\tilde{H})$ such that
$\rho_H(Y) = \tilde{Y}$. From the fact that $\tilde{H}$ acts freely on $Y$ it follows immediately that for every
$x=(H,\om) \in Y$ we have 
$$
N_G(H)_x = G_x = H.
$$
 
4.\ 
Now pick any point $y_0 = (H,\om_0) \in Y$ and let
$X={\cls}{Gy_0} \subset \Theta$. 
If $x_1=(H_1,\om_1)$ is a continuity point of the map $x \mapsto G_x$
(from $X$ to $\mathcal{S}(G)$),
and $g_i y_0 \to x_1$ then $G_{g_iy_0} = g_iHg_i^{-1} \to G_{x_1} = H_1
= \pi(x_1)\in Z$. Thus on the dense $G_\del$ set $X_c\subset X$ of continuity points
of the map $x \mapsto G_x$, this map coincides with $\pi$.
In other words $Z$ is the stability system of $X$.
\end{proof}

\begin{prob}\label{min-stab}
Can Proposition \ref{real} be strengthened to provide a minimal $(X,G)$ with $Z$ as its
stability system ?
\end{prob}

\end{document}